\documentclass[12pt]{amsart}

\usepackage{amssymb,amsthm}

\newtheorem{thm}{Theorem}
\newtheorem{theorem}{Theorem}[section]

\newtheorem{proposition}[theorem]{Proposition}

\def\irr#1{{\rm  Irr}(#1)}

\title[On a theorem of Artin]{On a theorem of Artin and the dimension of the space spanned by the rational valued characters of a group}

\begin{document}
	
\author[M. L. Lewis]{Mark L. Lewis}
\address{Department of Mathematical Sciences, Kent State University, Kent, OH 44242, USA}
\email{lewis@math.kent.edu}
	
\begin{abstract}  
In this paper, we sharpen a theorem of Artin to show that for a finite group, the dimension of the subspace of class functions spanned by the rational valued characters equals the number of conjugacy classes of cyclic subgroups.	
\end{abstract}
	


	
\maketitle


\section{Introduction}
	
Throughout this paper, all groups are finite except where noted.  In this paper, we consider the vector space of complex valued class functions on a group.  A {\it class function} on a group $G$ is a function defined on $G$ that is constant on the conjugacy classes of $G$ that takes on complex values.  If $\eta$ and $\zeta$ are two class functions of the group $G$, then $\eta+ \zeta$ is defined by $\left( \eta + \zeta\right) (g) = \eta (g) + \zeta (g)$ for all $g \in G$ and if $\alpha$ is a complex number, then $\alpha \eta$ is defined by $(\alpha \eta) (g) = \alpha (\eta (g))$.  In this way, the set of class functions of $G$ is a vector space over the complex numbers. 

In Theorem 2.8 of Isaacs' classic character text book \cite{text}, when $G$ is a group, it is shown that the vector space of class functions on $G$ has $\irr G$ as a basis where $\irr G$ is the set of irreducible characters of $G$.  The set of all {\it characters} of $G$ form a subset of the class functions that is the set of linear combinations of irreducible characters whose coefficients are nonnegative integers that are not all $0$, and the set of {\it generalized characters} is the set of linear combinations of irreducible character whose coefficients are integers.  For the full story on the characters and irreducible characters of a group, including how they are defined, we suggest the reader consult \cite{text}.     

A class function or character of $G$ is rational valued if every value lies in the field of rational numbers.  In \cite{text}, he proves as Theorem 5.21, the following theorem of E. Artin:

\begin{thm}[Isaacs' version of Artin's theorem]\label{artin}
Let $\chi$ be a rational valued character of $G$.  Then
$$\chi = \sum \frac {a_H}{|{\bf N}_G (H):H|}	(1_H)^G,$$
where $H$ runs over the cyclic subgroups of $G$ and $a_H \in Z$.
\end{thm}

Interestingly, this theorem does not seem to be covered in many other character theory textbooks; although it is covered in Serre's book \cite{serre}.  The version covered in Serre's book is somewhat different than the one in Isaacs, and we discuss it at the end of this introduction.  In fact, we have found very few other publications that mention this result.  The two places where we have found mention of Artin's theorem are Theorem 1 of \cite{kramar} and Theorem 3 of \cite{yang}.  The paper \cite{turull} also applies Isaacs' version Artin's theorem, and Turull obtains a result in a similar vein to ours.  We note that after this paper was submitted and while it was undergoing the third round of refereeing, a preprint has appeared that proves Lemma 2.3 (see \cite{pre}).

Our purpose in this note is to sharpen Isaacs's version of Artin's theorem.  We begin by pinning down the dimension of the subspace of the class functions that is spanned by the rational valued characters.  Theorem \ref{artin} shows that every rational character is a rational linear combination of the characters induced from the principal characters of the cyclic subgroups.  This says characters induced from the principal characters of the cyclic subgroups from a spanning set of the set of class functions spanned by the rational characters.  

Note that principal characters of conjugate cyclic subgroups yield the same character when induced to the group.  Thus, one need only look at the representatives of the conjugacy classes of the cyclic subgroups.  Thus, in Isaacs' theorem, we need only consider a set of representatives of the conjugacy classes of the cyclic subgroups of the group.  Hence, the number of conjugacy classes of cyclic subgroups will be key.  This gives our main theorems.  One should compare this result with Corollary 4 of \cite{turull}.

\begin{thm} \label {dimension}
Let $G$ be a group.  Then the subspace of complex valued class functions of $G$ spanned by the rational valued characters of $G$ has dimension equal to the number of conjugacy classes of cyclic subgroups of $G$.
\end{thm} 

We now pin down a basis for the subspace of class functions spanned by rational valued characters.   To obtain this basis, we use the fact that the set of induced characters from the principal characters of a set of representatives of the conjugacy classes of the cyclic subgroups of the group is linearly independent.

\begin{thm} \label{basis}
Let $G$ be a group and let $C_1, \dots, C_n$ be cyclic subgroups of $G$.  Then the following are true:
\begin{enumerate}
\item $(1_{C_1})^G, \dots, (1_{C_n})^G$ spans the subspace of complex valued class functions of $G$ spanned by the rational valued characters if and only if every conjugacy class of cyclic subgroups of $G$ is represented among the $C_i$'s.
\item $(1_{C_1})^G, \dots, (1_{C_n})^G$ is a basis for the subspace of complex valued class functions of $G$ spanned by the rational valued characters if and only if the $C_i$'s is a set of representative for all the conjugacy classes of cyclic subgroups of $G$.
\end{enumerate}
\end{thm}

At this point, we want to mention the version of Artin's theorem that appears in Serre's book \cite{serre}.  In Section 9.2 of that book he states as a corollary: each character of $G$ is a linear combination with rational coefficients of characters induced by characters of cyclic subgroups of $G$.  Notice that unlike Isaacs, he is obtaining all characters of $G$ instead of just the rational characters of $G$ by inducing all of the characters of the cyclic subgroups instead of only the principal characters.  Also, one can note that this is a corollary of a more general result.   Rather than give the version of the more general result as stated in Serre, we think the version found in \cite{overflow} is a better fit with this paper, so we state it here:

\begin{thm}[Serre's version of Artin's theorem]
Let $G$ be a group and let $H_1, \dots, H_n$ be subgroups of $G$.  Then every element of $G$ is conjugate to an element in one of the $H_i$'s if and only if every character of $G$ is a linear combination of characters induced from the $H_i$'s with rational coefficients.
\end{thm}

With this in mind and using our refinement, we can make the following restatement of Isaacs version of Artin's theorem:

\begin{thm}[Restatement of Isaacs' version of Artin's Theorem] \label{restatement}
Let $G$ be a (finite) group.  Then every cyclic subgroup of $G$ is conjugate to one of $C_1, \dots, C_n$ if and only if every rational-valued character of $G$ is a $Q$-linear combination of the induced characters from the principal characters of the $C_i$'s.
\end{thm}


In addition, we would like to thank Yiftach Barnea for bringing this problem to our attention.  We note that Isaacs' and Serre's books are written at essentially the same time and do not refer to each other, so it seems likely that neither of them knew of the other's books when writing their books.  Thus, we believe that the two versions of the theorem were written independently of each other.  So far, we have not been able to find the original source for this theorem.  

Finally, we would also like to thank Geoffrey Robinson for bringing our paper to Alex Turull's attention and Alex for pointing out the existence of \cite{turull}.   As we noted above, Corollary 4 of \cite{turull} has a statement similar to ours for permutation representations instead of rational valued characters.  The key to Turull's proof is Corollary 3 of \cite{turull} and in the first line of that proof, he states that Isaacs' version of Artin's theorem implies that the characters induced from the representatives of the conjugacy classes of the cyclic subgroups form a basis for the space spanned by the rational characters.  

However, we say that Isaacs' version of Artin's theorem only implies that the characters induced from the representatives of the conjugacy classes of the cyclic subgroups span the space spanned by rational characters, but that it does not imply that these characters form a basis.  The main point of this note is to give an explicit proof of the linear independence of these characters, a property used in \cite{turull} without a separate proof.  

We would like to thank the referees for a careful reading of this paper and several helpful comments that helped us to improve the paper.  In particular, we thank the second referee for pointing us to the existence of \cite{pre} and who caused us to look closely at the Turull paper and come to terms with what was bothering us with that paper.

\section{Proofs}

We will see that the work for proving these two theorems is contained in the proof of Theorem 5.21 of \cite{text}.  Note that the characteristic functions for the conjugacy classes are rational class functions and form a basis for the vector space of all class functions.   
In order to prove Theorem \ref{dimension}, we will rely heavily on the proof of Theorem 5.21 of \cite{text}.  Our proof is also influenced by the remarks before the statement of Theorem 3 and its proof in \cite{yang}.

\begin{proof}[Proof of Theorem \ref{dimension}]
Following the proof of Theorem 5.21 in \cite{text}, we define an equivalence relation $\equiv$ on $G$ by $x \equiv y$ if $\langle x \rangle$ is conjugate to $\langle y \rangle$.   Let $\mathcal {C}_1, \dots, \mathcal {C}_m$ be the distinct $\equiv$-classes of $G$ and let $\Phi_i$ be the characteristic function of $\mathcal {C}_i$ so that $\Phi_i (x) = 1$ if $x \in \mathcal {C}_i$ and $\Phi_i (x) = 0$ otherwise.  Choose a representative $x_i \in \mathcal {C}_i$; set $C_i = \langle x_i \rangle$ and $n_i = |C_i|$.  We have 
$$|\mathcal {C}_i| = |G:N_G (C_i)| \phi (n_i)$$
where $\phi$ is the Euler $\phi$ function.  Observe that $C_1, \dots, C_m$ are a set of representatives of the conjugacy classes of cyclic subgroups of $G$.  

It is mentioned in the proof of Theorem 5.21 of \cite{text} that Lemma 5.22 of \cite{text} implies that every rational valued character is an integer sum of the $\Phi_i$'s.  It is not difficult to see that the $\Phi_i$'s will be linearly independent.  Set $S = \{ |N_G (C_1)| \Phi_1, \dots, |N_G (C_m)|\Phi_m \}$.  The bulk of the proof of Theorem 5.21 of \cite{text} shows that $|N_G (C_i)| \Phi_i$ is a sum of induced characters from cyclic subgroups with rational coefficients.  Thus, the elements of $S$ lie in the subspace of class functions spanned by rational characters.  We see that $S$ spans this subspace and is linearly independent.  Thus, $S$ is a basis for the subspace of class functions spanned by the rational characters.  Since $S$ is in bijection with the set of all conjugacy classes of cyclic subgroups of $G$, this yields the conclusion. 
\end{proof}

We now work to prove Theorem \ref{basis}.

\begin{proof}[Proof of Theorem \ref{basis}]
Observe that if $C_i$ and $C_j$ are conjugate in $G$, then $(1_{C_i})^G = (1_{C_j})^G$.  By Artin's theorem (Theorem 5.21 of \cite{text}), we know the subspace of the vector space of class functions that spanned by the rational characters is spanned by the set of characters induced from the cyclic subgroups.  Hence, if the $C_i$'s contain representatives of all the conjugacy classes of cyclic subgroups, then the characters induced from the principal characters of the subgroups will span the subspace of the vector space of class functions spanned by the rational characters.  

If the $C_i$'s are a set of representatives of the conjugacy classes of cyclic subgroups, then by Theorem \ref{dimension}, the set of induced characters of the $1_{C_i}$'s is a spanning set of the subspace spanned by the rational characters that has the size of a basis, so the set of induced characters of the $C_i$'s will be a basis for the subspace.  

Conversely, suppose that the set of characters induced from the $1_{C_i}$'s spans the subspace spanned by the rational characters.   We see that the number of distinct induced characters will equal the number of different conjugacy classes of cyclic subgroups that are represented among the $C_i$'s.  It follows that the induced characters will span the subspace spanned by the rational characters only if all of the conjugacy classes of cyclic subgroups are represented.  Furthermore, if the induced characters form a basis, then we do not have any duplicates among the conjugacy classes, so the $C_i$'s must be a set of representatives of the conjugacy classes of cyclic subgroups of $G$.  
\end{proof}

We now give a quick proof of the restatement of Artin's theorem.

\begin{proof}[Proof of Theorem \ref{restatement}]
Suppose every cyclic subgroup of $G$ is conjugate to one of $C_1, \dots, C_n$.  If $C$ is a cyclic subgroup of $G$, then $C$ is conjugate to $C_i$ for some $i$, and $(1_C)^G = (1_{C_i})^G$, and we have the result by Theorem \ref{artin}.  Conversely, suppose every rational character is a $Q$-linear combination of the $(1_{C_i})^G$'s.  This implies that the rational characters all lie in the complex subspace spanned by the $(1_{C_i})^G$'s.  It follows that the complex subspace spanned by the $(1_{C_i})^G$ has the same dimension as the complex subspace spanned by the rational valued characters.  By Theorem \ref{basis}, we can conclude that every conjugacy class of cyclic subgroups must be conjugate to one of the $C_i$'s.  This proves the result.  
\end{proof}

\section{Another viewpoint}

In this section, we want to take another viewpoint on Artin's theorem.  Let $G$ be a group.  We let $n$ be the exponent of $G$ and let $F = Q(e^{2\pi i/n})$.  By a theorem of Brauer (Theorem 10.3 of \cite{text}), we know that $F$ is a splitting field for $G$.  (See Chapter 9 of \cite{text} for the definition of splitting field.)  For our purposes, this means that every (irreducible) character of $G$ has values in $F$.  Hence, $\Gamma = {\rm Gal} (F/Q)$ acts on the irreducible characters of $G$.  If $\chi \in \irr G$, then $Q (\chi)$ is the subfield of $F$ obtained by adjoining the values of $\chi$ to $Q$.  It is shown in Lemma 9.17 of \cite{text} that the orbit of $\chi$ under $\Gamma$ is $\{ \chi^\sigma \mid \sigma \in {\rm Gal} (Q(\chi)/Q) \}$.  

Given $\chi \in \irr G$, we define $\chi^{\rm rat} = \sum_{\sigma \in {\rm Gal} (Q(\chi)/Q)} \chi^\sigma$.  It is not difficult to see that $\chi^{\rm rat}$ is rational valued.  We say that $\chi^{\rm rat}$ is the {\it rationalization} of $\chi$.  If $\chi_1, \dots, \chi_l$ are a set of representatives of the orbits of $\Gamma$ on $\irr G$, then it is known and not difficult to show that every rational character is a sum of $(\chi_1)^{\rm rat}, \dots, (\chi_l)^{\rm rat}$.  (I.e., every rational character is a linear combination of the $(\chi_i)^{\rm rat}$ with nonnegative integer coefficients that are not all $0$.)  With this point of view, we obtain the following characterization of Artin's result:

\begin{proposition} \label{count}
Let $G$ be a group.  Then the number of conjugacy classes of cyclic subgroups of $G$ equals the number of orbits of the action $\Gamma = {\rm Gal}(F/Q)$ on $\irr G$.  
\end{proposition}

On the surface, the fact that the number of conjugacy classes of cyclic subgroup equals the number of orbits of $\Gamma$ on $\irr G$ seems surprising.  On the other hand, using the hint for Problem 2.12 of \cite{text}, we can define an action of $\Gamma$ on the conjugacy classes of $G$.  Let $n$ be the exponent of $G$.  One can show for $\sigma \in \Gamma$ that there exists $m$ with $(m,n) = 1$ so that $\chi^\sigma (g) = \chi (g^m)$ for all $g \in G$.  If we write $g^G$ for the conjugacy class of $g$, then we can define $(g^G)^\sigma = (g^m)^G$.  It is not difficult to see that as $\sigma$ runs over all the elements of $\Gamma$ that $g^m$ will run over all the generators for $\langle g \rangle$ and so, the union of the $(g^G)^\sigma$ will the union of the conjugacies of the generators of $\langle g \rangle$.  We can view the conjugacy class of the cyclic subgroup $\langle g \rangle$ to be the ``rationalization'' of $g^G$.  With this in mind, perhaps the equality in Proposition \ref{count} is not so surprising.

\section{Statement}

No use of generative AI was used in the preparation of this paper.


\end{document}